\documentclass[11pt]{amsart}
\usepackage{amssymb,pb-diagram}
\usepackage{amsmath}
\usepackage{amsthm,amssymb,amsfonts}
\usepackage{amssymb,amscd}
\usepackage{epic,eepic,epsfig}
\usepackage{latexsym}
\usepackage{xy}
\xyoption{all}

\newtheorem{Thm}{Theorem}[section]

\newtheorem{Cor}[Thm]{Corollary}
\newtheorem{Lem}[Thm]{Lemma}
\newtheorem{Prop}[Thm]{Proposition}

\theoremstyle{definition}
\newtheorem{Def}[Thm]{Definition}

\theoremstyle{remark}

\def \cal{\mathcal}

\def\eps{\varepsilon}

\def\Ndb{\mathbb N}

\def\Rdb{\mathbb R}

\def\Sz{\text{Sz}}

\newcommand{\bib}{\bibitem}

\def\Sz{\text{Sz }}

\begin{document}

\title{Three-space property for asymptotically uniformly smooth renormings}

\author{P.A.H. Brooker$^*$}
\address{$^*$Universit\'{e} de Franche-Comt\'{e}, Laboratoire de Math\'{e}matiques UMR 6623,
16 route de Gray, 25030 Besan\c{c}on Cedex, FRANCE
\vskip 0cm
$^*$46/116 Blamey Crescent, Campbell ACT 2612, AUSTRALIA}
\email{philip.a.h.brooker@gmail.com}

\author{G. Lancien$^{**}$}

\address{$^{**}$Universit\'{e} de Franche-Comt\'{e}, Laboratoire de Math\'{e}matiques UMR 6623,
16 route de Gray, 25030 Besan\c{c}on Cedex, FRANCE.}
\email{gilles.lancien@univ-fcomte.fr}

\begin{abstract} We prove that if $Y$ is a closed subspace of a Banach space $X$ such that $Y$ and $X/Y$ admit an equivalent asymptotically uniformly smooth norm, then $X$ also admits an equivalent asymptotically uniformly smooth norm. The proof is based on the use of the Szlenk index and yields a few other applications to renorming theory.
\end{abstract}

\subjclass[2010]{46B20}

\maketitle

\markboth{P.A.H. BROOKER AND G. LANCIEN}{}

\section{Introduction, notation}

A three space property for Banach spaces is a property, let us call it $(P)$, which satisfies the following: a Banach space $X$ has property $(P)$ whenever there is a closed subspace $Y$ of $X$ such that the Banach spaces $Y$ and $X/Y$ have property $(P)$. We refer to the book by J. Castillo and M. Gonz\'{a}lez \cite{CG} for a thorough exposition of the properties of Banach spaces that are known to be, or not to be, three space properties. Let us mention the fundamental result by Enflo, Lindenstrauss and Pisier who proved in \cite{ELP} that being isomorphic to a Hilbert space is not a three space property. On the other hand, it is known that reflexivity (see \cite{KS}) and super-reflexivity are three space properties. It then follows from Enflo's famous renorming theorem (\cite{E}) that having an equivalent uniformly convex norm (or having an equivalent uniformly smooth) norm is a three space property.

The aim of this note is to prove that having an equivalent asymptotically uniformly smooth norm is also a three space property (Theorem \ref{main}). The proof is based on the use of the Szlenk index. We shall derive a few applications to other types of renorming.

We now need to to define the properties of norms that will be
considered in this paper. For a Banach space $(X,\|\ \|)$ we
denote by $B_X$ the closed unit ball of $X$ and by $S_X$ its unit
sphere. For $t>0$, $x\in S_X$ and $Y$ a closed linear subspace of
$X$, we define
$$\overline{\rho}(t,x,Y)=\sup_{y\in S_Y}\|x+t y\|-1\ \ \ \ {\rm and}\ \ \
\ \overline{\delta}(t,x,Y)=\inf_{y\in S_Y}\|x+t y\|-1.$$ Then
$$\overline{\rho}_X(t)=\sup_{x\in S_X}\ \inf_{{\rm
dim}(X/Y)<\infty}\overline{\rho}(t,x,Y)\ \ \ \ {\rm and}\ \ \ \
\overline{\delta}_X(t)=\inf_{x\in S_X}\ \sup_{{\rm
dim}(X/Y)<\infty}\overline{\delta}(t,x,Y).$$ The norm $\|\ \|$ is said to be
{\it asymptotically uniformly smooth} (in short AUS) if
$$\lim_{t \to 0}\frac{\overline{\rho}_X(t)}{t}=0.$$
It is said to be {\it asymptotically uniformly convex} (in short
AUC) if
$$\forall t>0\ \ \ \ \overline{\delta}_X(t)>0.$$

These moduli have been first introduced by V. Milman in \cite{M}.

\medskip\noindent
Similarly in $X^*$ there is a weak$^*$-modulus of
asymptotic uniform convexity defined by
$$ \overline{\delta}_X^*(t)=\inf_{x^*\in S_{X^*}}\sup_{E}\inf_{y^*\in S_E}\{\|x^*+ty^*\|-1\},$$

\noindent where $E$ runs through all weak$^*$-closed subspaces of
$X^*$ of finite codimension.

\noindent The norm of $X^*$ is said to be {\it weak$^*$
asymptotically convex} (in short w$^*$-AUC) if
$$\forall t>0\ \ \ \ \overline{\delta}_X^*(t)>0.$$

The following proposition is elementary. Its proof can for
instance be found in \cite{D}.

\begin{Prop}\label{duality} Let $X$ be a Banach space. Then $\|\ \|_X$ is
asymtotically uniformly smooth if and only if $\|\ \|_{X^*}$ is
weak$^*$ asymptotically uniformly convex.
\end{Prop}

Note that more precisely, it is proved in \cite{GKL} that for a separable Banach space  $\overline{\delta}_X^*(\ )$ and $\overline{\rho}_X(\ )$ are equivalent to the dual Young's function of each other (see \cite{M} for the case when $X$ has a shrinking basis).

\medskip\noindent {\bf Examples.}

\smallskip (1) If $X=(\sum_{n=1}^\infty F_n)_{\ell_p}$, $1\le p<\infty$ and the $F_n$'s are finite
dimensional, then
$\overline{\rho}_X(t)=\overline{\delta}_X(t)=(1+t^p)^{1/p}-1$.

\smallskip (2) If $X=(\sum_{n=1}^\infty F_n)_{c_0}$, $1\le p<\infty$ and the $F_n$'s are finite
dimensional, then for all $t\in (0,1]$,
$\overline{\rho}_{X}(t)=0$.

\medskip We shall also refer to a notion defined by R. Huff in
\cite{H}.

\begin{Def} The norm $\|\ \|$ of a Banach space $X$ is {\it nearly
uniformly convex} (in short NUC) if for any $\eps>0$ there exists
$\delta>0$ such that conv$\{x_n,\ n\ge 1\}\cap (1-\delta)B_X\neq
\emptyset$, whenever $(x_n)_{n=1}^\infty$ is an $\eps$-separated
sequence in $B_X$.
\end{Def}

Although it is not exactly stated in these terms, the following is
essentially proved in \cite{H}

\begin{Thm}\label{huff} {\bf (Huff 1980)} Let $X$ be a Banach space. Then $\|\ \|_X$ is NUC if
and only if $\|\ \|_X$ is AUC and $X$ is reflexive.
\end{Thm}

Note that $\ell_1$ is a non reflexive AUC Banach space.

\medskip Then, in \cite{P} S. Prus characterized the notion dual to
NUC in the following way.

\begin{Thm}\label{prus} {\bf (Prus 1989)} Let $X$ be a Banach space. Then $\|\
\|_{X^*}$ is NUC if and only if  for every $\eps> 0$ there is an
$\eta > 0$ such that, for all $t\in [0, \eta]$ and each basic
sequence $(x_n)_n$ in $B_X$, there is $k > 1$ such that $\|x_1
+tx_k\|\le 1+\eps t$.

\noindent Such a norm is naturally called {\it nearly uniformly
smooth} (in short NUS).

\noindent Finally $\|\ \|_X$ is NUS if and only if $X$ is
reflexive and $\|\ \|_X$ is AUS.

\end{Thm}

\medskip The Szlenk index is the fundamental object related to the existence of
equivalent asymptotically uniformly smooth norms. Let us first
define it.

\noindent Consider a real Banach space $X$ and $K$ a
weak$^*$-compact subset of $X^*$. For $\eps>0$ we let $\cal V$ be
the set of all relatively weak$^*$-open subsets $V$ of $K$ such
that the norm diameter of $V$ is less than $\eps$ and
$s_{\eps}K=K\setminus \cup\{V:V\in\cal V\}.$ Then we define
inductively $s_{\eps}^{\alpha}K$ for any ordinal $\alpha$ by
$s^{\alpha+1}_{\eps}K=s_{\eps}(s_{\eps}^{\alpha}K)$ and
$s^{\alpha}_{\eps}K={\displaystyle
\cap_{\beta<\alpha}}s_{\eps}^{\beta}K$ if $\alpha$ is a limit
ordinal. We denote by $B_{X^*}$ the closed unit ball of $X^*$. We
then define $\text{Sz}(X,\eps)$ to be the least ordinal $\alpha$
so that $s_{\eps}^{\alpha}B_{X^*}=\emptyset,$ if such an ordinal
exists. Otherwise we write $\text{Sz}(X,\eps)=\infty.$ The {\it
Szlenk index} of $X$ is finally defined by
$\text{Sz}(X)=\sup_{\eps>0}\text{Sz}(X,\eps)$. In the sequel
$\omega$ will denote the first infinite ordinal and $\omega_1$ the
first uncountable ordinal.

The seminal result on AUS renormings is due to H. Knaust, E. Odell
and T. Schlumprecht (\cite{KOS}). In fact they obtained the
following theorem, which contains much more information on the
structure of these spaces.

\begin{Thm} {\bf (Knaust-Odell-Schlumprecht 1999)} Let $X$ be a separable Banach
space such that Sz$(X)\leq \omega$ (or equivalently
such that Sz$(X,\eps)$ is finite for all $\eps>0$). Then there
exist a Banach space $Z=Y^*$ with a boundedly complete finite
dimensional decomposition $(H_j)$ and $p\in [1,+\infty)$ so that
$X^*$ embeds isomorphically (norm and weak$^*$) into $Z$ and
\begin{equation}\label{upper}
\|\sum z_j\|^p \geq \sum \|z_j\|^p\ \ {\rm for\ all\ block\ bases}\ (z_j)\ {\rm
of}\ (H_j).
\end{equation}
\end{Thm}

Then they deduced:

\begin{Cor}\label{ukktfae} Let $X$ be a separable Banach space. The following assertions are
equivalent.

(i) $X$ admits an equivalent AUS norm.

(ii) Sz$(X)\leq \omega$.

(iii) $X$ admits an equivalent norm, whose dual norm is w$^*$-AUC
with a power type modulus.

(iv) There exist $C>0$ and $p\in [1,+\infty)$ such that: $\forall
\eps>0$ Sz$(X,\eps)\leq C\eps^{-p}$.
\end{Cor}

This was extended by M. Raja in \cite{R} to the non separable
case.

\medskip In \cite{GKL}, the relationship between the different
exponents involved in the above statement are precisely described.

\begin{Thm} Let $X$ be a separable Banach space with
$\text{Sz}(X)\leq\omega$ and define the Szlenk power type of $X$
to be
$$p_X:=\inf\{q\geq 1,\ \sup_{\eps>0} \eps^q\text{Sz}(X,\eps)<\infty\}.$$
Then
$$p_X=\inf\{q\geq 1,\ \text{there\ is\ an\ equiv.\ norm}\ |\ |\ \text{on}\ X,
\ \ \exists c>0\ \forall \eps>0,\ \overline{\delta}_{|\ |}^*(t)\geq ct^q\}.$$
\end{Thm}

Let us recall that  $\overline\rho_X$ is equivalent to the dual Young function of
$\overline{\delta}_X^*$. In particular, if there exist $c>0$ and
$q>1$ so that $\overline{\delta}_X^*(t)\geq ct^q$, then there
exists $C>0$ such that $\overline\rho_X(t)\le Ct^p$, where $p$ is
the conjugate exponent of $q$.

\section{The three-space problem and the Szlenk index}

In \cite{La} it is proved that the condition ``Sz$(X)<\omega_1$''
is a three-space property. In fact it is shown more precisely that
if $Y$ is a closed subspace of a Banach space $X$ then $\Sz(X)\le
\Sz(X/Y).\Sz(Y)$. Our main objective will be to improve this
result as follows:

\begin{Prop}\label{key} Let $X$ be a Banach space such that $\Sz(X)<\omega_1$
and let $Y$ be a closed
subspace of $X$. Then there exists a constant $C\ge 1$ such that
$$\forall \eps>0\ \ \ \Sz(X,\eps)\le
\Sz(X/Y,\frac{\eps}{C}).\Sz(Y,\frac{\eps}{C}).$$
\end{Prop}

\begin{proof} Since the Szlenk index of a Banach space, when it is countable, is
determined by its separable subspaces (see \cite{La}), we may and
will assume that $X$ is separable. Then we start with the
following elementary Lemma.

\begin{Lem}\label{elementary} Let $F$ be a weak$^*$-compact subset of $X^*$. Then,
for any countable ordinal $\alpha$ and any $\eps>0$:
$$s_\eps^\alpha(F+\frac{\eps}{6}B_{X^*}) \subseteq
s_{\frac{\eps}{6}}^\alpha(F)+\frac{\eps}{6}B_{X^*}.$$
\end{Lem}

\begin{proof} This will be done by transfinite induction. The
statement is clearly true for $\alpha=0$. Let now $\alpha\ge 1$
and assume that our statement is true for all $\beta<\alpha$.

\noindent If $\alpha=\beta+1$ is a successor ordinal. Let us
consider $x^*\in s_\eps^\alpha(F+\frac{\eps}{6}B_{X^*})$. Then
there is a sequence $(x_n^*)$ in
$s_\eps^\beta(F+\frac{\eps}{6}B_{X^*}) \subseteq
s_{\frac{\eps}{6}}^\beta(F)+\frac{\eps}{6}B_{X^*}$ so that $x_n^*
\buildrel {w^*}\over {\longrightarrow} x^*$ and $\|x_n^*-x^*\|\ge
\frac{\eps}{2}$. So we can write $x_n^*=u_n^*+v_n^*$ with $u_n^*
\in s_{\frac{\eps}{6}}^\beta(F)$ and $v_n^* \in
\frac{\eps}{6}B_{X^*}$. By extracting a subsequence, we may as
well assume that $u_n^* \buildrel {w^*}\over {\longrightarrow}
u^*$ and $v_n^* \buildrel {w^*}\over {\longrightarrow} v^*$.
Clearly $u^*+v^*=x^*$ and $v^*\in \frac{\eps}{6}B_{X^*}$. Finally,
$$\|u^*-u_n^*\|\ge \|x^*-x_n^*\|-\|v^*-v_n^*\|\ge
\frac{\eps}{2}-\frac{2\eps}{6}=\frac{\eps}{6}.$$ This concludes the successor case.

\noindent Assume now that $\alpha$ is a countable limit ordinal
and $x^*\in s_\eps^\alpha(F+\frac{\eps}{6}B_{X^*})$. Then for any
$\beta<\alpha$, $x^*\in
s_{\frac{\eps}{6}}^\beta(F)+\frac{\eps}{6}B_{X^*}$. Since $\alpha$
is countable, we can pick an increasing sequence of ordinals
$(\beta_n)$ so that $\beta_n<\alpha$ and $\sup_n \beta_n=\alpha$.
Then, for all $n\in \Ndb$, we can write $x^*=u_n^*+v_n^*$ with
$u_n^* \in s_{\frac{\eps}{6}}^{\beta_n}(F)$ and $v_n^* \in
\frac{\eps}{6}B_{X^*}$. By extracting a subsequence, we may again
assume that $u_n^* \buildrel {w^*}\over {\longrightarrow} u^*$ and
$v_n^* \buildrel {w^*}\over {\longrightarrow} v^*$. Then
$u^*+v^*=x^*$, $v^*\in \frac{\eps}{6}B_{X^*}$ and $u^*\in
\bigcap_{n\ge
1}s_{\frac{\eps}{6}}^{\beta_n}(F)=s_{\frac{\eps}{6}}^{\alpha}(F)$.
The proof of this lemma is finished.

\end{proof}

Our next step is to show:

\begin{Lem}\label{quotientmap} Let $Q$ be the canonical quotient map from $X^*$ onto
$X^*/Y^{\perp}$. For $\eps>0$, we denote
$\gamma_\eps=\Sz(X/Y,\frac{\eps}{18})$. Then
$$Q(s_\eps^{\gamma_\eps.\alpha}(B_{X^*}))\subseteq
s_{\frac{\eps}{6}}^\alpha (B_{X^*/Y^{\perp}}).$$
\end{Lem}

\begin{proof} This will again be done by transfinite induction.
The case $\alpha=0$ is clear. So assume that $\alpha\ge 1$ and
that the conclusion of the above lemma holds for all
$\beta<\alpha$.

\noindent Let $\alpha=\beta+1$ be a successor ordinal and let
$x^*\in X^*$ be such that $Qx^*\notin s_{\frac{\eps}{6}}^{\beta+1}
(B_{X^*/Y^{\perp}})$. We want to show that $x^*\notin
s_\eps^{\gamma_\eps.(\beta+1)}(B_{X^*})$. Therefore, we may assume
that $x^* \in s_\eps^{\gamma_\eps.\beta}(B_{X^*})$ and so, by
induction hypothesis, that $Qx^*\in s_{\frac{\eps}{6}}^{\beta}
(B_{X^*/Y^{\perp}})$. Hence, there exists a weak$^*$-neighborhood
$W$ of $Qx^*$ such that the diameter of $W\cap
s_{\frac{\eps}{6}}^{\beta} (B_{X^*/Y^{\perp}})$ is less than
$\frac{\eps}{6}$. By reducing it, we may assume that
$$W=\bigcap_{i=1}^n \{z\in X^*/Y^{\perp},\ \langle y_i,z\rangle
>a_i\}\ \ {\rm with}\ y_i\in Y\ {\rm and}\ a_i\in \Rdb\}.$$ Then we consider
$$V=\bigcap_{i=1}^n \{\xi^*\in X^*,\ \langle y_i,\xi^*\rangle
>a_i\}.$$
The set $V$ is a weak$^*$-neighborhood of $x^*$. Let now $\xi^*\in
V\cap s_\eps^{\gamma_\eps.\beta}(B_{X^*})$. Then, by induction
hypothesis, $Q\xi^* \in W\cap s_{\frac{\eps}{6}}^{\beta}
(B_{X^*/Y^{\perp}})$ and $\|Q\xi^*-Qx^*\|<\frac{\eps}{6}$. It
follows that $\xi^*\in x^*+3B_{Y^\perp}+\frac{\eps}{6}B_{X^*}$.
Therefore, we have shown that
$$\big(V\cap s_\eps^{\gamma_\eps.\beta}(B_{X^*})\big) \subseteq
F+\frac{\eps}{6}B_{X^*},$$ where $F=x^*+3B_{Y^\perp}$. It follows
now from Lemma \ref{elementary} that for any ordinal $\delta$:
$$s_\eps^\delta\big(V\cap s_\eps^{\gamma_\eps.\beta}(B_{X^*})\big)
\subseteq s_{\frac{\eps}{6}}^\delta(F)+\frac{\eps}{6}B_{X^*}.$$
Finally an easy use of translations and homogeneity shows that
$s_{\frac{\eps}{6}}^{\gamma_\eps}(F)=\emptyset$, which implies
that $x^* \notin s_\eps^{\gamma_\eps.(\beta+1)}(B_{X^*})$ and
finishes the successor case.

Assume now that $\alpha$ is a limit ordinal. By definition of the
ordinal multiplication, $\gamma_\eps.\alpha=\sup_{\beta<\alpha}
\gamma_\eps.\beta$. The induction is then clear.

\end{proof}

Proposition \ref{key} follows clearly from Lemma
\ref{quotientmap}.
\end{proof}

It is well known and quite elementary that for a Banach space $X$,
either $\Sz(X)=\infty$ (equivalently, $X$ is not an Asplund space)
or there exists an ordinal $\alpha$ such that
$\Sz(X)=\omega^\alpha$ (see \cite{La2} for details). It is then natural
to wonder whether, for any ordinal $\alpha$,  the condition
``$\Sz(X)= \omega^\alpha$'' is a three-space property. Proposition
\ref{key} does not allow us to obtain this conclusion. However, we
can deduce the following.

\begin{Cor}\label{sz3sp} For any countable ordinal $\alpha$, the condition ``$\Sz(X)\le
\omega^{\omega^\alpha}$'' is a three-space property.
\end{Cor}

\begin{proof} Let $X$ be a Banach space and $Y$ be a subspace of
$X$ so that $\Sz(Y)\le \omega^{\omega^\alpha}$ and $\Sz(X/Y)\le
\omega^{\omega^\alpha}$. Note that, for any $\eps>0$, and any
Banach space $Z$, either $\Sz(Z,\eps)=\infty$ or $\Sz(Z,\eps)$ is
a successor ordinal (it follows clearly from the use of
weak$^*$-compactness). Now, it is an elementary fact from ordinal
arithmetic that whenever $\alpha,\ \beta$ and $\gamma$ are
ordinals such that $\beta <\omega^{\omega^\alpha}$ and $\gamma
<\omega^{\omega^\alpha}$, then $\beta
.\gamma<\omega^{\omega^\alpha}$. This finishes our proof.
\end{proof}

\noindent {\bf Remark.}

(1) A closer look at the above proof actually
allows to show that both conditions ``$\Sz(X)=
\omega^{\omega^\alpha}$'' and ``$\Sz(X)< \omega^{\omega^\alpha}$''
are three-space properties.

\smallskip (2) It should also be noted that, except for $\alpha=0$, Corollary \ref{sz3sp} and the above remark follow from the estimate obtained in \cite{La}. The only new information is that the condition $\Sz(X)\le \omega$ is a three space property. This is however crucial, since this will have some consequences in renorming theory as we will see in detail in the last section.

\section{Applications to renorming theory}

We can now state our main result.

\begin{Thm}\label{main}
Let $Y$ be a closed linear subspace of a separable Banach space
$X$. Assume that $Y$ and $X/Y$ admit an equivalent AUS renorming.
Then $X$ admits an equivalent AUS norm.

\indent More precisely, let $p,q\in [1,+\infty)$ and $c>0$ be such
that
$$\forall t>0\ \ \ \overline{\delta}_Y^*(t)\ge c\eps^p\ \ {\rm
and}\ \ \overline{\delta}_{X/Y}^*(t)\ge c\eps^q.$$ Then, for all
$\delta>0$, there exists an equivalent norm $|\ |$ on $X$ and a
constant $\gamma>0$ such that

$$\forall t>0\ \ \ \overline{\delta}_{|\ |}^*(t)\ge
\gamma\eps^{p+q+\delta}.$$
\end{Thm}

\begin{proof} Under these assumptions, it is an easy and classical
fact that there is a constant $C>0$ such that
$$\Sz(Y,\eps)\le C\eps^{-p}\ \ \ {\rm and}\ \ \ \Sz(X/Y,\eps)\le
C\eps^{-q},\ \ \ \eps\in(0,1).$$ We now deduce from Proposition \ref{key} that
there exists $C'>0$ so that
$$\Sz(X,\eps)\le C'\eps^{-(p+q)}.$$
Finally, the conclusion follows from Theorem 4.8 in \cite{GKL}.

\end{proof}

\noindent {\bf Remark.} The fact that AUS renormability is also a
three-space property in the non separable case can be viewed as a
consequence of Proposition \ref{key} and Raja's non separable
version \cite{R} of the renorming theorem of Knaust, Odell and
Schlumprecht \cite{KOS}. In order to get the above quantitative
estimate, one needs to extend Theorem 4.8 of \cite{GKL} to the
non-separable settings. This is feasible by using nets instead of
sequences and the corresponding uncountably branching trees and
tree maps.

\begin{Cor}\label{NUC} Let $Y$ be a closed linear subspace of a Banach space
$X$. Assume that $Y$ and $X/Y$ admit an equivalent NUS
(respectively NUC) renorming. Then $X$ admits an equivalent NUS
(respectively NUC) norm.
\end{Cor}

\begin{proof} In both cases, the assumptions imply that $Y$ and
$X/Y$ are reflexive. Thus, by Krein-\v{S}mulyan's Theorem \cite{KS}, $X$ is
also reflexive. Then, by Theorem \ref{prus} and a simple duality
argument, we only need to prove the NUS case. The proof now
follows from the fact that a norm on $X$ is NUS if and only if it
is AUS and $X$ is reflexive and from Theorem \ref{main}

\end{proof}

We shall now consider another asymptotic property of norms. It was introduced by S. Rolewicz in \cite{Ro} and is now called property $(\beta)$ of Rolewicz. For its definition, we shall use a characterization due to D. Kutzarova \cite{Ku1}.

\begin{Def} An infinite-dimensional Banach space $X$ is said to have property $(\beta)$ if for any $t\in (0,1]$,  there exists $\delta>0$ such that for any $x$ in $B_X$ and any $t$-separated sequence $(x_n)_{n=1}^\infty$ in $B_X$, there exists $n\ge 1$ so that
$$\frac{\|x-x_n\|}{2}\le 1-\delta.$$
For a given $t\in (0,1]$, we denote $\overline{\beta}_X(t)$ the supremum of all $\delta\ge 0$ so that the above property is satisfied.
\end{Def}

We shall need the following result:

\begin{Thm} Let $X$ be a separable Banach space. The following assertions are equivalent.

(i) $X$ admits an equivalent norm with property $(\beta)$.

(ii) $X$ is reflexive, $\text{Sz}(X)\le \omega$ and $\text{Sz}(X^*)\le \omega$.

(iii) $X$ is reflexive, admits an equivalent AUS norm and an equivalent AUC norm.

(iv)  $X$ is reflexive and admits an equivalent norm which is simultaneously AUS and AUC.

\end{Thm}

This is detailled in \cite{DKLR}, where the proof is presented as a gathering of some previous results. We refer the reader to \cite{DKLR} for the proper references. These references include in particular results from \cite{Ku0}, \cite{P}, \cite{OS} and \cite{La}.

\medskip Then the following is an immediate consequence of Corollary \ref{NUC}

\begin{Cor} Let $X$ be a separable Banach space and $Y$ be a closed linear subspace of $X$. Assume that $Y$ and $X/Y$ admit an equivalent norm with property $(\beta)$. Then $X$ admits an equivalent norm with property $(\beta)$.

\end{Cor}

\end{document}